\documentclass[12pt]{article}
\usepackage{amsmath}
\usepackage{amsthm}
\usepackage{amsfonts}
\usepackage{amssymb}
\usepackage{esint}
\usepackage{eucal}
\usepackage{mathrsfs}
\usepackage{graphicx,graphics}
\numberwithin{equation}{section}
\usepackage{graphicx,graphics}
\usepackage{pdfsync}
\usepackage{multirow}
\usepackage{endnotes}
\usepackage{epstopdf}
\def \dis {\displaystyle}

\def \limi {\;\mathop{\longrightarrow}_{n\to\infty}\;}

\def \confai {-\kern -.5em\rightharpoonup}
\def \cqfd {\hfill$\Box$}
\def \div{\mbox{\rm div}}

\def \al {\alpha}
\def \be {\beta}

\def \Ga {\Gamma}

\def \De {\Delta}
\def \ep {\varepsilon}

\def \Om {\Omega}
\def \la {\lambda}

\def \ph {\varphi}

\def \si {\sigma}
\def \Si {\Sigma}

\def \ZZ {\mathbb Z}

\def \RR {\mathbb R}

\def \beq {\begin{equation}}
\def \eeq {\end{equation}}
\def \ba {\begin{array}}
\def \ea {\end{array}}
\def \bs {\bigskip}
\def \ms {\medskip}
\def \ss {\smallskip}
\def \ecart {\noalign{\medskip}}
\newtheorem{Thm}{Theorem}[section]
\newtheorem{Cor}[Thm]{Corollary}
\newtheorem{Pro}[Thm]{Proposition}

\newtheorem{Conj}[Thm]{Conjecture}
\newtheorem{Adef}[Thm]{Definition}

\newtheorem{Arem}[Thm]{Remark}
\newenvironment{Rem}{\begin{Arem}\rm}{\end{Arem}}
\newtheorem{Aappl}[Thm]{Application}

\newtheorem{Aexa}[Thm]{Example}
\newenvironment{Exa}{\begin{Aexa}\rm}{\end{Aexa}}
\newtheorem{Anot}[Thm]{Notation}

\def \refe #1.{(\ref{#1})}
\def \reff #1.{figure~\ref{#1}}
\def \refs #1.{Section~\ref{#1}}
\def \refss #1.{Subsection~\ref{#1}}
\def \refD #1.{Definition~\ref{#1}}
\def \refT #1.{Theorem~\ref{#1}}
\def \refL #1.{Lemma~\ref{#1}}
\def \refC #1.{Corollary~\ref{#1}}
\def \refP #1.{Proposition~\ref{#1}}
\def \refPt #1.{Properties~\ref{#1}}
\def \refR #1.{Remark~\ref{#1}}
\def \refA #1.{Application~\ref{#1}}
\def \refE #1.{Example~\ref{#1}}
\def \refN #1.{Notation~\ref{#1}}
\newcounter{marnote}

\title{Isotropic realizability of electric fields \\Êaround critical points}
\author{Marc Briane\footnote{INSA de Rennes \& IRMAR, FRANCE -- mbriane@insa-rennes.fr}}
\begin{document}
\maketitle
\begin{abstract}
In this paper we study the isotropic realizability of a given regular gradient field $\nabla u$ as an electric field, namely when $\nabla u$ is solution of the equation $\div\left(\si\nabla u\right)=0$ for some isotropic conductivity $\si>0$. The case of a function $u$ without critical point was investigated in~\cite{BMT} thanks to a gradient flow approach. The presence of a critical point needs a specific treatment according to the behavior of the dynamical system around the point. The case of a saddle point is the most favorable and leads us to a characterization of the local isotropic realizability through some boundedness condition involving the laplacian of $u$ along the gradient flow. The case of a sink or a source implies a strong maximum principle under the same boundedness condition. However, when the critical point is not hyperbolic the isotropic realizability is not generally satisfied even piecewisely in the neighborhood of the point. The isotropic realizability in the torus for periodic gradient fields is also discussed in particular when the trajectories of the gradient system are bounded.
\end{abstract}
\par\bs\noindent
\section{Introduction}
The starting point of the present paper is the following issue: given a gradient field $\nabla u$ from $\RR^d$ into $\RR^d$, under which conditions $\nabla u$ is an isotropically realizable electric field, namely there exists an isotropic conductivity $\si>0$ such that $\div\left(\si\nabla u\right)=0$ in $\RR^d$\,?
A quite complete answer is given in \cite{BMT} when $u$ is regular and $\nabla u$ is periodic. On the one hand, assuming that $u\in C^3(\RR^d)$ and
\beq\label{Du­0}
\inf_{\RR^d}|\nabla u|>0,
\eeq
we may construct a continuous conductivity $\si$ in $\RR^d$ (this is a straightforward extension of Theorem~2.15 in \cite{BMT}).
In dimension two this implies that $\nabla u$ can be rectified globally in $\RR^2$ into a constant vector (see Theorem~\ref{thm.globrect} in the Appendix). In contrast the rectification theorem only applies locally for an arbitrary smooth non-vanishing vector field (see, {\em e.g.}, \cite{Arn}).
On the other hand, when $\nabla u$ is periodic, it is not always possible to derive a periodic conductivity under the sole condition~\refe{Du­0}..
Actually, the Theorem~2.17 of \cite{BMT} shows that $\nabla u$ is isotropically realizable in the torus under the extra assumption
\beq\label{bound}
\sup_{x\in\RR^d}\left|\,\int_0^{\tau(x)}\De u\big(X(s,x)\big)\,ds\,\right|<\infty,
\eeq
where $X$ is the gradient flow defined by
\beq\label{X}
\left\{\ba{ll}
X'(t,x)=\nabla u\big(X(t,x)\big), & \mbox{for }t\in\RR
\\ \ecart
X(0,x)=x, & 
\ea\right.
\eeq
and $\tau(x)$ is the time (unique by \refe{Du­0}.) for the flow $X(\cdot,x)$ to reach the equipotential $\{u=0\}$.
Moreover, the boundedness condition \refe{bound}. is also necessary to derive a periodic conductivity $\si$ in~$C^1(\RR^d)$.
It is then natural to ask what happens when condition \refe{Du­0}. does not hold due to the existence of critical points.
This paper provides partial answers to this question.
\par\bs
We essentially study the realizability of a gradient field in the neighborhood of an isolated critical point. The originality of this local problem lies in the following fact: Surprisingly, the boundedness assumption \refe{bound}. which in \cite{BMT} is useless for the realizability around any regular point but necessary for the global realizability in the torus, turns out to be crucial in its local version to derive the realizability around an isolated critical point. Consider a function $u\in C^2(\RR^d)$, for $d=2,3$, and a point $x_*\in\RR^d$ such that
\beq\label{isocri}
x_*\mbox{ is an isolated critical point of $u$}.
\eeq
The issue is to know when $\nabla u$ is isotropically realizable in the neighborhood of $x_*$. To this end the gradient system \refe{X}. plays an important role as in \cite{BMT}. In view of \refe{bound}. we establish a strong connection between the isotropic realizability of $\nabla u$ around the critical point $x_*$ and the boundedness of the function
\beq\label{intLapuX}
(t,x)\longmapsto\int_0^t\De u\big(X(s,x)\big)\,ds.
\eeq
More precisely, we study the local isotropic realizability according to the nature of the critical point $x_*$. The following cases are investigated:
\begin{itemize}
\item $x_*$ is a saddle point, {\em i.e.} $\nabla^2 u(x_*)$ is invertible with both positive and negative eigenvalues;
\item $x_*$ is a sink (resp. a source), {\em i.e.} $\nabla^2 u(x_*)$ is negative (resp. positive) definite, or more generally, $x_*$ is a stable point for the gradient system \refe{X}. (see definition~\refe{stable}. below);
\item $x_*$ is not a hyperbolic point, {\em i.e.} $\det\big(\nabla^2 u(x_*)\big)=0$ and $x_*$ is not stable.
\end{itemize}
\par
When $x_*$ is a saddle point, we prove (see Theorem~\ref{thm.saddle}) that the local isotropic realizability is (in some sense) equivalent to a local version of the bound~\refe{bound}. combined with \refe{isocri}.. In Section~\ref{ss.app} the two-dimensional example $u(x,y)=f(x)+g(y)$ illustrates the sharpness of the boundedness condition which also implies that $\De u(x_*)=0$ (see Remark~\ref{rem.Lapux*}). Moreover, in spite of its simplicity this example shows that  the question of realizability around a critical point is rather delicate (see Section~\ref{ss.prel} and Proposition~\ref{pro.fg}). In particular, the condition $\De u(x_*)=0$ turns out to be not sufficient to derive the isotropic realizability in the neighborhood of the point~$x_*$. Indeed, we construct a function $u\in C^2\big([-1,1]^2\big)$ which admits $(0,0)$ as a saddle point satisfying $\De u(0,0)=0$, but the gradient of which is not isotropically realizable around~$(0,0)$ (see Proposition~\ref{pro.fg}~$iii)$).
\par
In the case of a sink or a source the boundedness of the function \refe{intLapuX}. leads us to a strong maximum principle (see Theorem~\ref{thm.stable}). When $x_*$ is not a hyperbolic point, the situation is much more intricate. We study a two-dimensional example where the isotropic realizability is only satisfied in some regions around $x_*$, again in connection with condition~\refe{bound}. (see Proposition~\ref{pro.deg}).
\par
We conclude the paper with the isotropic realizability problem in the torus. The natural extension of \cite{BMT} (Theorem~2.17) is that any regular periodic gradient field which vanishes at isolated points is isotropically realizable provided that the boundedness condition \refe{bound}. holds (see Conjecture~\ref{conj1}). We prove this result under the additional assumption that the trajectories of \refe{X}. are bounded (see Theorem~\ref{thm.globsing}), and we illustrate it by Proposition~\ref{pro.ui}. At this level the dimension two is quite particular. Indeed, by virtue of \cite{AlNe} (see also \cite{BaMN} for the non-periodic case) a non-zero periodic gradient field which is isotropically realizable with a smooth periodic conductivity does not vanish in $\RR^2$, and the trajectories of the gradient system \refe{X}. are then unbounded (see Remark~\ref{rem.d=2}). This shows that the boundedness of the trajectories together with the presence of critical points is a reasonable assumption in the periodic framework. 
\section{The case of a saddle point}
\subsection{A preliminary remark}\label{ss.prel}
Let $u\in C^2(\RR^d)$, for $d\geq 2$, and let $x_*$ be a non-degenerate critical point of $u$, namely $\nabla u(x_*)=0$ and the hessian matrix $\nabla^2 u(x_*)$ is invertible. By Morse's lemma (see, {\em e.g.}, \cite{Mil} Lemma~2.2) there exist a $C^2$-diffeomorphism $\Phi$ from an open neighborhood $V_*$ of $x_*$ onto an open set $W_*$, and an integer $m\in\{0,\dots,d\}$ independent of $V_*$ such that
\beq\label{M}
v(y):=(u\circ\Phi^{-1})(y)-u(x_*)=-\sum_{j=1}^m y_j^2+\sum_{j=m+1}^d y_j^2\,,\quad\forall\,y\in W_*.
\eeq
Assume that the gradient field $\nabla v$ is isotropically realizable with a smooth conductivity $\tau>0$ in $W_*$, namely $\div\big(\tau\nabla v\big)=0$ in $W_*$. This implies that $\De v\big(\Phi(x_*)\big)=0$ by virtue of Remark~\ref{rem.Lapux*} below, and thus $d=2m$. Conversely, if $d=2m$, the function $v$ is harmonic and $\nabla v$ is isotropically realizable with the conductivity $\tau\equiv 1$. Hence, by the change of variables $y=\Phi(x)$ we get that
\beq\label{Phisi}
\div\left(\si\nabla u\right)=0\;\;\mbox{in }V_*,\quad\mbox{with}\quad\si:=\det\left(\nabla\Phi\right)\nabla\Phi^{-1}(\nabla\Phi^{-1})^T.
\eeq
Therefore, when $d=2m$, the gradient $\nabla u$ is realizable with the smooth conductivity $\si$ of \refe{Phisi}., but which is {\em a priori} anisotropic. We will see that the isotropic realizability of $\nabla u$ around the point $x_*$ is more delicate to obtain. In particular the equality $\De u(x_*)=0$ is generally not a sufficient condition of isotropic realizability contrary to the case of the quadratic function~\refe{M}..
\subsection{The general result}
Let $u\in C^2(\RR^d)$, for $d=2,3$. Consider a critical point $x_*$ of $u$, which is also a saddle point of $u$ in the sense that 
\beq\label{saddle}
\nabla^2 u(x_*)\mbox{ has a non-zero determinant with both positive and negative eigenvalues.}
\eeq
Without loss of generality we may also assume that $u(x_*)=0$.
\par
The Hadamard-Perron Theorem (see, {\em e.g.}, \cite{AnAr} p.~56, and \cite{HSD} Section~8.3) claims that in some compact neighborhood $K_*$ of $x_*$, containing no extra critical point, there exist two smooth invariant manifolds of the flow \refe{X}.: $\Ga^s$ of dimension $k$ and $\Ga^u$ of dimension $d-k$ (which are smooth curves for in dimension two), such that
\begin{itemize}
\item $\Ga^s\cap\Ga^u=\{x_*\}$,
\item $\Ga^s$ contains only stable trajectories of \refe{X}. (converging to $x_*$ as $t\to\infty$) and $\Ga^u$ contains only unstable trajectories (converging to $x_*$ as $t\to-\infty$),
\item for any $x\in K_*\setminus(\Ga^s\cup\Ga^u)$, the trajectory $X(t,x)$ leaves $K_*$ as $|t|$ increases.
\end{itemize}
Note that we have for any $(t,x)\in\RR\times(K_*\setminus\{x_*\})$ with $X(t,x)\in{\rm int}\left(K_*\right)$,
\beq\label{dtuX}
{\partial\over\partial t}\big(u(X(t,x)\big)=\left|\nabla u\big(X(t,x)\big)\right|^2>0,
\eeq
which implies that the function $u\big(X(\cdot,x)\big)$ is increasing. 
Hence, $u\big(X(\cdot,x)\big)$ is negative for $x\in\Ga^s\setminus\{x_*\}$ (due to $u\big(X(t,x)\big)\to 0$ as $t\to\infty$) and is positive for $x\in\Ga^u\setminus\{x_*\}$ (due to $u\big(X(t,x)\big)\to 0$ as $t\to-\infty$). So, each connected component of the equipotential $\{u=0\}\setminus\{x_*\}$ lies between two connected components of $\Ga^s\setminus\{x_*\}$ and $\Ga^u\setminus\{x_*\}$.
Therefore, it is relevant to assume that there exists an open neighborhood $Q_*$ of $x_*$, with $Q_*\subset{\rm int}\left(K_*\right)$, such that
\beq\label{taux}
\forall\,x\in Q_*\setminus(\Ga^s\cup\Ga^u),\ \exists\,\tau(x)\in\RR,\quad
\left\{\ba{l}
X(t,x)\in {\rm int}\left(K_*\right)\mbox{ for }t\in[0,\tau(x)]
\\ \ecart
X\big(\tau(x),x\big)\in\{u=0\}\cap{\rm int}\left(K_*\right),
\ea\right.
\eeq
namely each trajectory $X(\cdot,x)$, for $x\in Q_*\setminus(\Ga^s\cup\Ga^u)$, meets the equipotential $\{u=0\}$ at an interior point of $K_*$.
Moreover, since by \refe{dtuX}. the function $u\big(X(\cdot,x)\big)$ is increasing in the open interval containing $\tau(x)$:
\beq
\big\{t\in\RR\,:\,X(t,x)\in{\rm int}\left(K_*\right)\big\},
\eeq
the time $\tau(x)$ is uniquely determined (see \reff{fig1}. for a picture of the situation).
\par
On the other hand, if $u\in C^3(\RR^d)$ the gradient flow $X$ defined by \refe{X}. is in $C^1(\RR\times\RR^d)$. By virtue of the implicit functions theorem, the regularity of $X$ and \refe{dtuX}. imply that $\tau$ belongs to $C^1\big(Q_*\setminus(\Ga^s\cup\Ga^u)\big)$. Then, we may define the function $w$ in $Q_*\setminus(\Ga^s\cup\Ga^u)$ by
\beq\label{wtau}
w(x):=\int_0^{\tau(x)}\De u\big(X(s,x)\big)\,ds,\quad\mbox{for }x\in Q_*\setminus(\Ga^s\cup\Ga^u),
\eeq
which belongs to $C^1\big(Q_*\setminus(\Ga^s\cup\Ga^u)\big)$. If $u$ is only in $C^2(\RR^d)$, then $w$ does not belong necessarily to $C^1\big(Q_*\setminus(\Ga^s\cup\Ga^u)\big)$. However, in the sequel we will assume that
\beq\label{wC1}
u\in C^2(\RR^d)\quad\mbox{and}\quad w\in C^1\big(Q_*\setminus(\Ga^s\cup\Ga^u)\big).
\eeq
\par
We have the following result:
\begin{Thm}\label{thm.saddle}
Let $u\in C^2(\RR^d)$, for $d=2,3$, and let $x_*\in\RR^d$ be such that conditions \refe{isocri}., \refe{saddle}., \refe{taux}. and \refe{wC1}. hold with an open neighborhood $Q_*$ and a compact neighborhood $K_*$ of $x_*$ satisfying $Q_*\subset{\rm int}\left(K_*\right)$.
\begin{itemize}
\item[$i)$] Assume that $w$ belongs to $L^\infty(Q_*)$. Then, $\nabla u$ is isotropically realizable in $Q_*$, with a positive conductivity $\si$ such that $\si,\si^{-1}\in L^\infty(Q_*)\cup C^1\big(Q_*\setminus(\Ga^s\cup\Ga^u)\big)$.
\item[$ii)$] Conversely, assume that $\nabla u$ is isotropically realizable in the interior of $K_*$, with a conductivity $\si$ such that $\si,\si^{-1}\in L^\infty\big({\rm int}\left(K_*\right)\big)\cup C^1\big({\rm int}\left(K_*\right)\setminus(\Ga^s\cup\Ga^u)\big)$. Then, the function $w$ belongs to $L^\infty(Q_*)$.
\end{itemize}
\end{Thm}
\begin{Rem}\label{rem.Lapux*}
In the part $i)$ of Theorem~\ref{thm.saddle}, the regularity of the conductivity $\si$ implies that $\De u(x_*)=0$.
Indeed, when $\sigma\in C^1(Q_*)$ we have
\beq
\div\left(\si\nabla u\right)(x_*)=\nabla\si(x_*)\cdot\nabla u(x_*)+\si(x_*)\,\De u(x_*)=\si(x_*)\,\De u(x_*)=0.
\eeq
More generally, having in mind the part $ii)$ of Theorem~\ref{thm.saddle}, the boundedness of the function $w$ \refe{wtau}. also implies that $\De u(x_*)=0$. Indeed, taking up a heuristic point of view, consider a sequence $x_n$ in $K_*\setminus(\Ga^s\cup\Ga^u)\big)$ which converges to some $x\in(\Ga^s\cup\Ga^u)\setminus\{x_*\}$, say $x\in\Ga^s\setminus\{x_*\}$. The sequence $|\tau(x_n)|$ tends to $\infty$, since the trajectory $X(\cdot,x)$ does not intersect the equipotential $\{u=0\}$. Then, under some suitable assumption satisfied by $\De u\big(X(s,\cdot)\big)$ in the neighborhood of $x_*$ and the fact that $X(t,x)$ tends to $x_*$ as $t\to\infty$, we have
\beq
0=\lim_{n\to\infty}{w(x_n)\over\tau(x_n)}=\lim_{n\to\infty}\fint_0^{\tau(x_n)}\!\!\De u\big(X(s,x_n)\big)\,ds
=\lim_{n\to\infty}\fint_0^{\tau(x_n)}\!\!\De u\big(X(s,x)\big)\,ds=\De u(x_*).
\eeq
This will be rigorously checked in the particular case of subsection~\ref{ss.app} (see Proposition~\ref{pro.fg} $i)$).
However, Proposition~\ref{pro.fg} $iii)$ will show that the equality $\De u(x_*)=0$ is not sufficient to get the isotropic realizability of $\nabla u$ around the point $x_*$.
\end{Rem}
\noindent
{\bf Proof of Theorem~\ref{thm.saddle}.}
\par\ms\noindent{\it Proof of $i)$.}
Assume that $w\in L^\infty(Q_*)$. Then, the conductivity defined by $\si:=e^w$ and $\si^{-1}$ belong to $L^\infty(\Om)\cap C^1\big(Q_*\setminus(\Ga^s\cup\Ga^u)\big)$. Moreover, following the proof of Theorem~2.14 in~\cite{BMT} but with the weaker regularity \refe{wC1}., the equation $\div\left(\si\nabla u\right)=0$ holds in each connected component of $Q_*\setminus(\Ga^s\cup\Ga^u)$. For the reader's convenience we recall the main steps of the proof:
\par
Let $\Om$ be a connected component of $Q_*\setminus(\Ga^s\cup\Ga^u)$. First note that, since $u\in C^2(\RR^d)$, the flow $X$ belongs to $C^0(\RR\times\RR^d)$. Hence, for any $x\in\Om$ and $t$ close to $0$ the trajectory $X(t,x)$ remains in $\Om$. Then, due to the semi-group property satisfied by the gradient flow $X$ \refe{X}. combined with the uniqueness of $\tau$, we have for any $x\in\Om$ and any $t$ close to $0$,
\beq
\tau\big(X(t,x)\big)=\tau(x)-t,
\eeq
which, using successively the semi-group property and the change of variable $r=s+t$, yields 
\beq
\ba{ll}
\dis w\big(X(t,x)\big)=\int_0^{X(t,x)}\De u\big(X(s,X(t,x)\big)\,ds & \dis =\int_0^{\tau(x)-t}\De u\big(X(s+t,x)\big)\,ds
\\ \ecart
& \dis =\int_t^{\tau(x)}\De u\big(X(r,x)\big)\,dr.
\ea
\eeq
Taking the derivative with respect to $t$, which is valid since $w\in C^1(\Om)$ by \refe{wC1}., it follows that
\beq
{\partial\over\partial t}\left[w\big(X(t,x)\big)\right]=\nabla w\big(X(t,x)\big)\cdot\nabla u\big(X(t,x)\big)=-\,\De u\big(X(t,x)\big).
\eeq
Therefore, for $t=0$ we get that
\beq
\nabla w(x)\cdot\nabla u(x)=-\,\De u(x),\quad\forall\,x\in\Om,
\eeq
which combined with $\si=e^w$ implies that
\beq
\div\left(\si\nabla u\right)=\si\left(\nabla w\cdot\nabla u+\De u\right)=0\quad\mbox{in }\Om.
\eeq
\par
It thus remains to prove that the equation is satisfied in $Q_*$. To this end we distinguish the cases $d=2$ and $d=3$.
\par
First assume that $d=2$. Then, $\Ga^s$ and $\Ga^u$ are two smooth curves in $K_*$ which only intersect at the point $x_*$. Consider $\ep>0$ such that the open ball $B(x_*,\ep)$ centered on $x_*$ and of radius $\ep$, is contained in $Q_*$. Let $\ph_\ep$ be a Lipschitz function in $Q_*$, with compact support in $Q_*$ and $\ph_\ep\equiv 0$ in $B(x_*,\ep)$. Let $\Om$ be a connected component of $Q_*\setminus(\Ga^s\cup\Ga^u)$. Note that $\si\nabla u$ is a divergence free vector-valued function in $L^\infty(\Om)^d$, thus has a trace on $\partial\Om$. Then, integrating by parts we have
\beq
\int_\Om\si\nabla u\cdot\nabla\ph_\ep\,dx=\int_{\partial\Om\cap\Ga^s}\si\nabla u\cdot n\,\ph_\ep\,ds+\int_{\partial\Om\cap\Ga^u}\si\nabla u\cdot n\,\ph_\ep\,ds,
\eeq
where $n$ denotes the normal outside to $\partial\Om$. However, since $\Ga^s$ and $\Ga^u$ are trajectories of the gradient system \refe{X}., at each point $x$ of $\partial\Om\cap(\Ga^s\cup\Ga^u)$ the normal $n$ is orthogonal to the tangent vector $X'(0,x)=\nabla u\big(X(0,x)\big)=\nabla u(x)$ to the curve $\Ga^s$ or $\Ga^u$. Therefore, we get that
\beq\label{intphep}
\int_\Om\si\nabla u\cdot\nabla\ph_\ep\,dx=0.
\eeq
Let $\ph\in C^\infty_c(Q_*)$, and set $\ph_\ep:=\ph\,v_\ep$ where $v_\ep$ is defined by
\beq\label{vep}
v_\ep(x):=\left\{\ba{cl}
0 & \mbox{if }|x|<\ep
\\ \ecart
\dis 2-{2\ep\over|x|} & \mbox{if } \ep\leq |x|\leq 2\ep
\\ \ecart
1 & \mbox{if }|x|>2\ep,
\ea\right.
\eeq
so that $\ph_\ep$ is a Lipschitz function in $Q_*$, with compact support in $Q_*$ and $\ph_\ep\equiv 0$ in $B(x_*,\ep)$.
Using that $v_\ep$ converges strongly to $1$ in $W^{1,1}(\Om)$ and $\si\nabla u\in L^\infty(\Om)^d$, we deduce from equality \refe{intphep}. that
\beq
\int_\Om\si\nabla u\cdot\nabla\ph\,dx=\int_\Om\si\nabla u\cdot\nabla\big(\ph\,(1-v_\ep)\big)\,dx+
\int_\Om\si\nabla u\cdot\nabla\ph_\ep\,dx=o(1),
\eeq
for each connected component $\Om$ of $Q_*\cap(\Ga^s\cup\Ga^u)$. It follows that $\si\nabla u$ is divergence free in~$Q_*$.
\par
In dimension three one of the manifold $\Ga^s$ or $\Ga^u$ is a curve $\Ga$, while the other one is a smooth surface $\Si=\{f=0\}$ composed of trajectories \refe{X}.. Consider a Lipschitz function $\ph_\ep$ with compact support in $Q_*$, which is zero in a tube of radius $\ep$ surrounding the curve $\Ga$ and containing the ball $B(x_*,\ep)$. For any $x\in\Si$ and $t$ close to $0$, the derivative of $f\big(X(t,x)\big)=0$ at $t=0$ yields $\nabla f\big(X(0,x)\big)\cdot X'(0,x)=\nabla f(x)\cdot\nabla u(x)=0$. Hence, the normal at each point $x$ of $\Si$ is orthogonal to $\nabla u(x)$, which again leads us to equality \refe{intphep}.. Therefore, passing to the limit as $\ep\to 0$ we obtain that $\si\nabla u$ is divergence free in $Q_*$.
\par\ms\noindent{\it Proof of $ii)$.}
Conversely, assume that in each connected component $\Om$ of ${\rm int}\left(K_*\right)\setminus(\Ga^s\cup\Ga^u)$, there exists $w_\Om\in L^\infty(\Om)\cap C^1(\Om)$ such that $\div\left(e^{w_\Om}\nabla u\right)=0$ in $\Om$. Then, we have
\beq
\nabla w_\Om(x)\cdot\nabla u(x)+\De u(x)=0,\quad\forall\,x\in\Om.
\eeq
This combined with \refe{taux}. implies that for any $x\in\Om\cap Q_*$,
\beq
\ba{ll}
\dis w_\Om(x)-w_\Om\big(X(\tau(x),x)\big) & \dis =-\int_0^{\tau(x)}{\partial \over\partial s}\left[w_\Om\big(X(s,x)\big)\right]ds
\\ \ecart
& \dis =\int_0^{\tau(x)}\De u\big(X(s,x)\big)\,ds=w(x),
\ea
\eeq
where $X(\tau(x),x)\in{\rm int}\left(K_*\right)$ by \refe{taux}..
Therefore, $w$ is in $L^\infty(\Om\cap Q_*)$ and thus in $L^\infty(Q_*)$. \cqfd
\begin{Rem}
If the function $w$ is not bounded in $Q_*$, the previous proof then shows that the function $\si:=e^w$ belongs to $C^1\big(Q_*\setminus(\Ga^s\cup\Ga^u)\big)$ and blows up near $\Ga^s\cup\Ga^u$. However, the equation $\div\left(\si\nabla u\right)=0$ still holds in $Q_*\setminus\{x\}$.
\end{Rem}
The following two-dimensional application shows that the boundedness of the function \refe{wtau}. is crucial for deriving the isotropic realizability of $\nabla u$ in the neighborhood of $x_*=(0,0)$, and that the sole equality $\De u(x_*)=0$ is not sufficient (see Proposition~\ref{pro.fg}). 
\subsection{Application}\label{ss.app}
To lighten the notations a point of $\RR^2$ is denoted by $(x,y)$ in this section.
Let $K_*:=[-\al,\al]^2$, for $\al>0$. Consider the function $u\in C^2(K_*)$ defined by
\beq\label{ufg}
u(x,y):=f(x)+g(y),\quad\mbox{for }(x,y)\in K_*,
\eeq
where the functions $f,g:[-\al,\al]\to\RR$ satisfy the following properties:
\beq\label{fg0}
f(0)=g(0)=0,\quad f'(0)=g'(0)=0,\quad f''(0)\,g''(0)<0,
\eeq
\beq\label{f'g'}
\left\{\ba{lll}
\forall\,x\in(0,\al], & f'(x)>0, & g'(x)<0,
\\ \ecart
\forall\,x\in[-\al,0), & f'(x)<0, & g'(x)>0.
\ea\right.
\eeq
Then, $u$ satisfies the assumptions of the general framework with the saddle point $(0,0)$ and the curves
\beq
\Ga^s=\{0\}\times[-\al,\al]\quad\mbox{and}\quad\Ga^u=[-\al,\al]\times\{0\}.
\eeq
The more delicate point to check is condition \refe{taux}.. To this end define the functions $F,G$ in $[-\al,\al]\setminus\{0\}$ by
\beq\label{FG}
\left\{\ba{lll}
\dis F(x):=\int_\al^x{dt\over f'(t)}<0, & \dis G(x):=\int_\al^x{dt\over g'(t)}>0, & \mbox{if }x\in(0,\al],
\\ \ecart
\dis F(x):=\int_{-\al}^x{dt\over f'(t)}<0, & \dis G(x):=\int_{-\al}^x{dt\over g'(t)}>0, & \mbox{if }x\in[-\al,0),
\ea\right.
\eeq
which are one-to-one from one of the two intervals $(0,\al]$ or $[-\al,0)$ onto one of the two intervals $(-\infty,0]$ or $[0,\infty)$.
Then, for any $(x,y)\in K_*$, the solution of \refe{X}. is given by
\beq\label{XFG}
X(t,x,y)=\big(F^{-1}(t+F(x)),G^{-1}(t+G(y))\big),\quad\mbox{for }t\in\big[\!-\!G(y),-F(x)\big],
\eeq
where $F^{-1},G^{-1}$ denote respectively the reciprocals of the functions $F,G$ restricted to each interval $(0,\al]$ or $[-\al,0)$.
Now, for $(x,y)\in K_*\setminus(\Ga^s\cup\Ga^u)$:
\begin{itemize}
\item if $u(x,y)<0$, consider
\beq
u\big(X(-F(x),x,y)\big)=f(\pm\al)+g\big(G^{-1}(-F(x)+G(y))\big)\mathop{\approx}_{x,y\to 0} f(\pm\al)+g(0)>0,
\eeq
with $+$ if $x>0$ and $-$ if $x<0$;
\item if $u(x,y)>0$, consider
\beq
u\big(X(-G(y),x,y)\big)=f\big(F^{-1}(-G(y)+F(x))\big)+g(\pm\al)\mathop{\approx}_{x,y\to 0} f(0)+g(\pm\al)<0,
\eeq
with $+$ if $y>0$ and $-$ if $y<0$.
\end{itemize}
Moreover, $u\big(X(\cdot,x,y)\big)$ is increasing on $\big[\!-\!G(y),-F(x)\big]$.
Then, defining the open set $Q_*$ by
\beq\label{Q*}
Q_*:=\left\{(x,y)\in{\rm int}\left(K_*\right)\,:\,
\ba{l}
f(\pm\al)+g\big(G^{-1}(-F(x)+G(y))\big)>0
\\*[.1cm]
g(\pm\al)+f\big(F^{-1}(-G(y)+F(x))\big)<0
\ea
\right\},
\eeq
with the convention $F(0)=-\infty$ and $G(0)=\infty$, we obtain that
\beq\label{tau(x,y)}
\forall\,(x,y)\in Q_*\setminus(\Ga^s\cup\Ga^u),\ \exists!\,\tau(x,y)\in\big(\!-\!G(y),-F(x)\big),\quad u\big(\tau(x,y),x,y\big)=0.
\eeq
Therefore \refe{taux}. holds.
\par
The next result shows that the isotropic realizability for $\nabla u$ is very sensitive to the boundedness of the function $w$ defined by \refe{wtau}.:
\begin{Pro}\label{pro.fg}
The function $w$ satisfies the regularity assumption \refe{wC1}.. Moreover, we have the following results:
\begin{itemize}
\item[$i)$] If $\nabla u$ is isotropically realizable with $\si$ such that $\si,\si^{-1}\in L^\infty_{\rm loc}(Q_*)\cap C^1\big(Q_*\setminus(\Ga^s\cup\Ga^u)\big)$, then $\De u(0,0)=0$.
\item[$ii)$] Let $u$ be a function satisfying \refe{ufg}., \refe{fg0}., \refe{f'g'}. with $\De u(0,0)=0$. Assume that the functions $F,G$ defined by \refe{FG}. satisfy
\beq\label{bouFG}
F-{\ln|\cdot|\over f''(0)}\quad\mbox{and}\quad G-{\ln|\cdot|\over g''(0)}\quad\mbox{are bounded in a neighborhood of }0.
\eeq
Then, $\nabla u$ is isotropically realizable with $\si$ such that $\si,\si^{-1}\in L^\infty_{\rm loc}(Q_*)\cap C^1\big(Q_*\setminus(\Ga^s\cup\Ga^u)\big)$.
\item[$iii)$] There exists a function $u\in C^2(K_*)$ satisfying \refe{ufg}., \refe{fg0}., \refe{f'g'}. with $\De u(0,0)=0$, such that for any open neighborhood $V_*$ of $(0,0)$, $\nabla u$ is not isotropically realizable with any conductivity $\si$ such that $\si,\si^{-1}\in L^\infty(V_*)\cap C^1\big(V_*\setminus(\Ga^s\cup\Ga^u)\big)$.
\end{itemize}
\end{Pro}
\begin{Rem}\label{f"g"0}
By \refe{fg0}. and \refe{FG}. we have the following asymptotics
\beq\label{eqFG}
F(x)\;\mathop{\sim}_{0}\;{\ln|x|\over f''(0)}\quad\mbox{and}\quad G(x)\;\mathop{\sim}_{0}\;{\ln|x|\over g''(0)}.
\eeq
Therefore, condition \refe{bouFG}. is stronger than \refe{eqFG}.. We will prove below that condition \refe{bouFG}. combined with the equality $\De u(0,0)=0$, or equivalently $f''(0)=-g''(0)$, implies the boundedness of the function $w$ \refe{wtau}..
Actually, the equality $f''(0)=-g''(0)$ without condition \refe{bouFG}. is not sufficient to ensure the isotropic realizability as shown in Proposition~\ref{pro.fg} $iii)$. However, if $f,g\in C^3\big([-\al,\al]\big)$, then \refe{bouFG}. is satisfied and thus the isotropic realizability holds.
\end{Rem}
\begin{Exa}
Consider the potential $u$ defined by
\beq
u(x,y):=\cos y-\cos x,\quad\mbox{for }(x,y)\in [-\pi,\pi]^2.
\eeq
The origin $(0,0)$ is a saddle point, which leads to the phase portrait of \reff{fig1}. with
\beq
\Ga^s=\{0\}\times[-\pi,\pi]\quad\mbox{and}\quad\Ga^s=[-\pi,\pi]\times\{0\}.
\eeq
A simple computation yields that the time $\tau(x,y)$ of \refe{tau(x,y)}. is given by
\beq
\dis \tau(x,y)={1\over 2}\,\ln\left|\,{\tan(y/2)\over\tan(x/2)}\,\right|,\quad\mbox{for }(x,y)\in (-\pi,\pi)^2\setminus(\Ga^s\cup\Ga^u).
\eeq
Hence, the function $w$ defined by \refe{wtau}. satisfies
\beq
\ba{ll}
w(x,y) & \dis =2\int_0^{\tau(x,y)}\left({1\over 1+\tan^2(x/2)\,e^{2t}}-{1\over 1+\tan^2(y/2)\,e^{-2t}}\right)dt
\\ \ecart
& \dis =\ln\left[{\big(1+\tan^2(x/2)\big)\big(1+\tan^2(y/2)\big)\over\big(1+|\tan(x/2)|\,|\tan(y/2)|\big)^2}\right],
\ea
\eeq
which is continuous in $(-\pi,\pi)^2$ and belongs to $C^1\big((-\pi,\pi)^2\setminus(\Ga^s\cup\Ga^u)\big)$. Therefore, from Theorem~\ref{thm.saddle} $i)$ we deduce that $\div\left(e^w\nabla u\right)=0$ in $(-\pi,\pi)^2$, which can be checked directly.
\end{Exa}
\begin{figure}
\centering
\vskip -3.cm
\includegraphics[scale=.7]{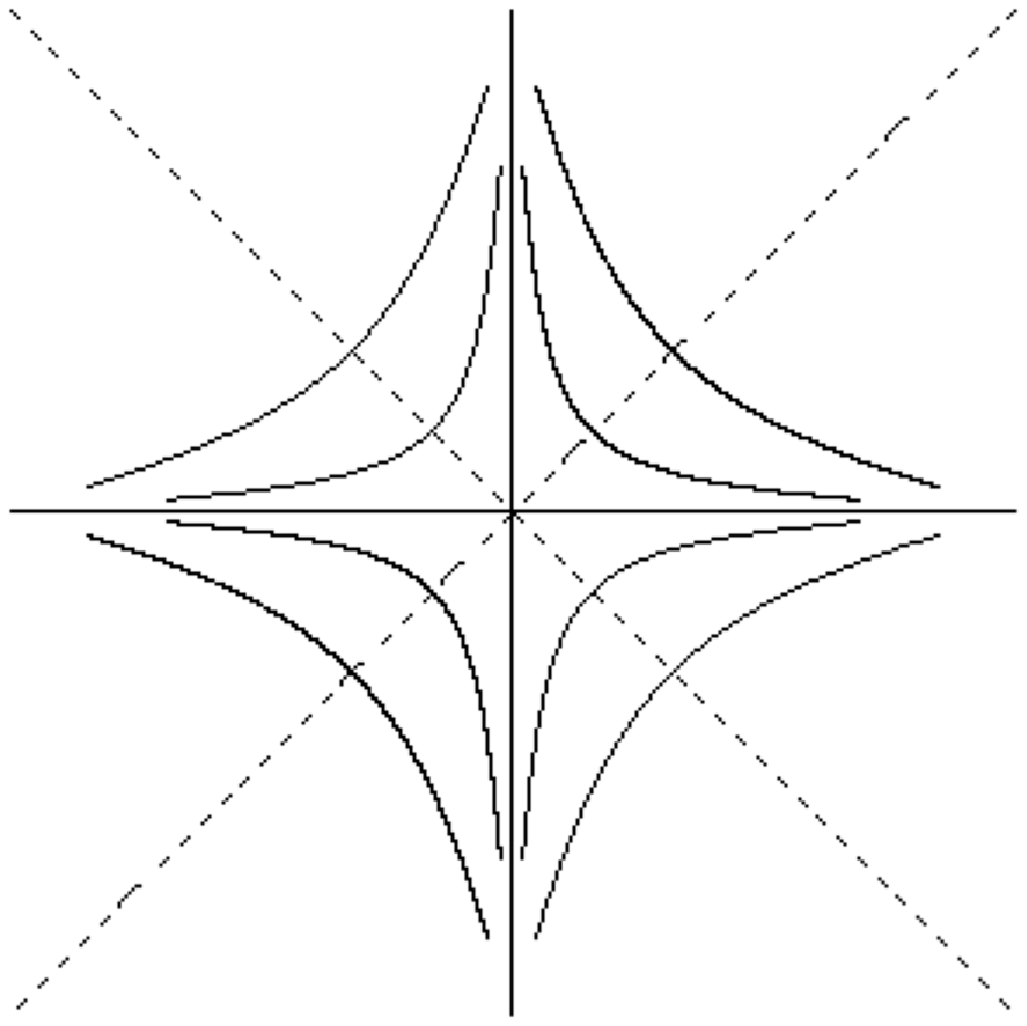}
\vskip -10.5cm
\caption{\it Phase portrait for $u=\cos y-\cos x$ around the saddle point $(0,0)$, with
the stable curve $\{x=0\}$, the unstable curve $\{y=0\}$, and the equipotential $\{u=0\}$ in dash intersecting the trajectories.}
\label{fig1}
\end{figure}
\noindent
{\bf Proof of Proposition~\ref{pro.fg}.}
Let $a,b$ be the functions defined in $Q_*\setminus(\Ga^s\cup\Ga^u)$ by
\beq\label{ab}
a(x,y):=F^{-1}\big(\tau(x,y)+F(x)\big),\;\; b(x,y):=G^{-1}(\tau(x,y)+G(y)\big),\quad(x,y)\in Q_*\setminus(\Ga^s\cup\Ga^u).
\eeq
From now on we simply denote $\tau(x,y)$, $a(x,y)$, $b(x,y)$ respectively by $\tau$, $a$, $b$.
By the definitions \refe{tau(x,y)}. of $\tau$, \refe{ufg}. of $u$ and \refe{XFG}. of $X$, we have
\beq\label{fgtau}
f\left[F^{-1}\big(\tau+F(x)\big)\right]+g\left[G^{-1}\big(\tau+G(y)\big)\right]=0\quad\mbox{for }(x,y)\in Q_*\setminus(\Ga^s\cup\Ga^u),
\eeq
or equivalently,
\beq\label{fgab}
f(a)+g(b)=0\quad\mbox{in }Q_*\setminus(\Ga^s\cup\Ga^u).
\eeq
Then, using \refe{XFG}., \refe{ab}. and the changes of variable $t=F^{-1}(s+F(x))$, $t=G^{-1}(s+G(y))$ in the formula \refe{wtau}. for $w$, we get that for any $(x,y)\in Q_*\setminus(\Ga^s\cup\Ga^u)$,
\beq\label{wfg}
\ba{ll}
w(x,y) & \dis =\int_0^\tau f''\big(F^{-1}(s+F(x))\big)\,ds+\int_0^\tau g''\big(G^{-1}(s+G(y))\big)\,ds
\\ \ecart
& \dis =\int_x^a{f''(t)\over f'(t)}\,dt+\int_y^b{g''(t)\over g'(t)}\,dt=\ln\left|\,{f'(a)\,g'(b)\over f'(x)\,g'(y)}\,\right|.
\ea
\eeq
Next, since the functions $f,g,F,G$ belong to $C^2\big([-\al,\al]\setminus\{0\}\big)$ and by \refe{FG}.
\beq
{\partial\over\partial t}\left[f\big(F^{-1}(t+F)\big)+g\big(G^{-1}(t+G)\big)\right]
=\left[f'\big(F^{-1}(t+F)\big)\right]^2+\left[g'\big(G^{-1}(t+G)\big)\right]^2>0,
\eeq
the implicit functions theorem implies that $\tau$ defined by \refe{fgtau}. belongs to $Q_*\setminus(\Ga^s\cup\Ga^u)$, so do the function $a,b$ defined by \refe{ab}.. Therefore, the formula~\refe{wfg}. shows that $w\in C^1\big(Q_*\setminus(\Ga^s\cup\Ga^u)\big)$, that is \refe{wC1}..
\par\ms\noindent
{\it Proof of $i)$.} Assume that $x\to 0$ with $x>0$, and $y\to y_0>0$. By \refe{FG}. and \refe{ab}. we also have $a,b>0$.
Then, we have $F(x)\to -\infty$, which implies that $\tau+F(x)\to -\infty$.
Indeed, if $\tau+F(x)\geq c$ for some constant $c<0$, then $a\geq F^{-1}(c)$ together with $\tau\to\infty$ and $\tau+G(y)\geq\tau\to\infty$. Therefore, $a\nrightarrow 0$ and $b\to 0$, which contradicts \refe{fgab}.. Similarly, we show that $\tau+G(y)\to \infty$. Therefore, we obtain that $a,b\to 0$, and as a consequence
\beq
{f'(a)\,g'(b)\over f'(x)\,g'(y)}\mathop{\sim}_{x\to 0}{f''(0)\,g''(0)\over g'(y_0)}\,{a\,b\over x}.
\eeq
This combined with the boundedness of $w$ \refe{wfg}. implies that
\beq\label{abx}
\ln a+\ln b-\ln x\quad\mbox{is bounded as}\quad x\to 0,\ y\to y_0.
\eeq
On the other hand, by \refe{eqFG}. we have
\beq
\tau+F(x)=F(a)\mathop{\sim}_{x\to 0}{\ln a\over f''(0)}\quad\mbox{and}\quad
\tau+G(y)=G(b)\mathop{\sim}_{x\to 0}{\ln b\over g''(0)},
\eeq
hence subtracting the two asymptotics and denoting $\la=-f''(0)/g''(0)$, it follows that
\beq\label{xab}
f''(0)\,\big(F(x)-G(y)\big)=\ln x+o(\ln x)=\ln a+o(\ln a)+\la\,\ln b+o(\ln b).
\eeq
However, from \refe{fgab}. we deduce that $f''(0)\,a^2\sim-\,g''(0)\,b^2$ and thus $\ln a\sim\ln b$.
Putting this in \refe{xab}. and using \refe{abx}. we get that
\beq
2\,\ln a+o(\ln a)=\left(1+\la\right)\ln a+o(\ln a),
\eeq
which yields $\la=1$, or equivalently $\De u(0,0)=0$.
\par\ms\noindent
{\it Proof of $ii)$.} For the sake of simplicity let us assume that $f''(0)=-g''(0)=1$.
By virtue of Theorem~\ref{thm.saddle} $i)$ we have to show that the function $w$ of \refe{wtau}. is bounded locally in $Q_*$.
However, since $w\in C^1\big(Q_*\setminus(\Ga^s\cup\Ga^u)\big)$, it is enough to prove that $w(x,y)$ remains bounded as $x\to 0$ or/and $y\to 0$. Let $(x,y)$ be a point of $Q_*\setminus(\Ga^s\cup\Ga^u)$. Without loss of generality we can assume that $x,y>0$, which by \refe{FG}. and \refe{ab}. also implies that $a,b>0$.
\par
First, assume that $x\to 0$ and $y\to y_0>0$. By the part $i)$ we have $a,b\to 0$. Then, taking into account \refe{fg0}. the asymptotics of $f,g$ at the point $0$ with $f''(0)=-g''(0)=1$, and formula \refe{wfg}. imply that
\beq\label{wasy}
w(x,y)\mathop{\sim}_{x\to 0}\ln\left|\,{a\,b\over x\,g'(y_0)}\,\right|.
\eeq
Moreover, due to \refe{bouFG}. there exists a constant $c>0$ such that
\beq\label{estFGln}
|F(x)-\ln x|+|\tau+F(x)-\ln a|+|G(y)+\ln y|+|\tau+G(y)+\ln b|\leq c,
\eeq
which implies that $|\ln a+\ln b-\ln x|$ is bounded as $x\to 0$, $y\to y_0>0$. Putting this estimate in \refe{wasy}. it follows that $w(x,y)$ is bounded as $x\to 0$, $y\to y_0>0$.
\par
The case where $x,y\to 0$ is quite similar using the asymptotic
\beq
w(x,y)\mathop{\sim}_{x,y\to 0}\ln\left({a\,b\over x\,y}\right),
\eeq
combined with estimate \refe{estFGln}..
\par\ms\noindent
{\it Proof of $iii)$.} Set $\al=1$. It is easy to check that there exists a unique one-to-one increasing $C^2$-function $F:(0,1]\to(-\infty,\be]$, defined implicitly by
\beq\label{Fimp}
\forall\,x\in(0,1],\;\;x=e^{\left(F(x)+\ln^2|F(x)|\right)}\quad\mbox{and}\quad F(1)=\be.
\eeq
An easy computation yields ($x\to 0^+$ means that $x\to 0$ with $x>0$)
\beq
\lim_{x\to 0^+}F(x)=-\infty,\quad\lim_{x\to 0^+}F'(x)=\infty\quad\mbox{and}\quad\lim_{x\to 0^+}x\,F'(x)=-\lim_{x\to 0^+}x^2\,F''(x)=1,
\eeq
so we may define the even function $f\in C^2\big([-1,1]\big)$ by
\beq\label{fF}
f(x):=\int_0^x{dt\over F'(t)},\quad\mbox{for }x\geq 0,
\eeq
which satisfies \refe{fg0}. with $f''(0)=1$. Moreover, consider the function $g$ defined by
\beq
g(y):=-{y^2\over 2},\quad\mbox{for }y\in[-1,1],
\eeq
so that $G(y)=-\ln y$ for $y>0$.
Then, by \refe{ab}. and \refe{wasy}., for $x>0$, $x\to 0$ and $y=y_0>0$, we have $b=y_0\,e^{-\tau}$ and
\beq\label{w(x,1)}
w(x,y_0)\mathop{\sim}_{x\to 0}\ln\left({a\,b\over x\,y_0}\right)=\ln a-\ln x-\tau.
\eeq
By \refe{Fimp}. we also have $a=e^{\left(\tau+F(x)+\ln^2|\tau+F(x)|\right)}$, hence
\beq\label{lnaxtau}
\ln a -\ln x-\tau=\ln^2|\tau+F(x)|-\ln^2|F(x)|=\ln\big|1+\tau/F(x)\big|\,\big(\ln|\tau+F(x)|+\ln|F(x)|\big).
\eeq
On the other hand, proceeding as in the proof of $i)$ we have $\tau\to\infty$ and $a,b\to 0$.
Then, the equality \refe{fgab}. $f(a)=b^2/2$ implies that $a\sim b=y_0\,e^{-\tau}$, hence
\beq
\ln\left(a\,e^\tau\right)=2\tau+F(x)+\ln^2|\tau+F(x)|=\ln y_0+o(1)\quad\mbox{with}\quad \tau+F(x)\to-\infty.
\eeq
It follows that
\beq
{2\tau+F(x)\over\tau+F(x)}=1+{1\over 1+F(x)/\tau}=o(1)\quad\mbox{and thus}\quad 2\tau\mathop{\sim}_{x\to 0}-F(x).
\eeq
Putting this asymptotic in \refe{lnaxtau}. together with \refe{w(x,1)}. we get that for any $y_0\in(0,1)$,
\beq
w(x,y_0)\mathop{\sim}_{x\to 0}-2\ln 2\,\ln|F(x)|\;\mathop{\longrightarrow}_{x\to 0}\;\infty.
\eeq
This combined with the fact that $w$ is continuous in $(0,1)^2$, shows that $w$ does not belong to $L^\infty(V_*)$ for any neighborhood $V_*$ of $(0,0)$. Therefore, the part $ii)$ of Theorem~\ref{thm.saddle} allows us to conclude. \cqfd
\section{The case of a sink or a source}
Let $u\in C^2(\RR^d)$, for $d\geq 2$. Consider a point $x_*\in\RR^d$ satisfying \refe{isocri}.. Assume that $x_*$ is stable for the gradient system~\refe{X}., namely there exist a compact neighborhood $K_*$ of $x_*$, containing no extra critical point, and a neighborhood $Q_*$ of $x_*$, with $Q_*\subset{\rm int}\left(K_*\right)$, such that
\beq\label{stable}
\big(\forall\,x\in Q_*,\ \forall\,t\geq 0,\;\; X(t,x)\in K_*\big)\quad\mbox{or}\quad\big(\forall\,x\in Q_*,\ \forall\,t\leq 0,\;\; X(t,x)\in K_*\big).
\eeq
In the first case $x_*$ is said to be positively stable, while in the second case it is said to be negatively stable.
\begin{Rem}
If $x_*$ is a sink (resp. source) point of the linearized system, namely $\nabla^2u(x_*)$ has only negative (resp. positive) eigenvalues, then $x_*$ is positively (resp. negatively) stable.
\end{Rem} 
\begin{Rem}\label{rem.stable}
If $x_*$ is a strict local extremum, then by Lyapunov's stability (see, {\em e.g.}, \cite{HSD} Theorem p.~194) $x_*$ is stable in the sense of definition \refe{stable}.. Conversely, if $x_*$ is stable, then it is asymptotically stable, namely each trajectory $X(t,x)$ for $x\in Q_*$, converges as $t\to\infty$ (resp. $t\to-\infty$) to the isolated critical point $x_*$ (see, {\em e.g.}, \cite{HSD} Proposition~p.~206). Therefore, since the function $u\big(X(\cdot,x)\big)$ is non-decreasing, $x_*$ is a local maximum (resp. minimum) of $u$.
\end{Rem}
In connection with Remark~\ref{rem.stable}, the following strong maximum principle holds:
\begin{Thm}\label{thm.stable}
Let $u\in C^2(\RR^d)$, for $d\geq 2$. Let $x_*$ be a positively (resp. negatively) stable point for \refe{X}., satisfying \refe{isocri}. with $\De u(x_*)=0$. Assume that there exists a constant $C_*>0$ such that
\beq\label{C*}
\forall\,t\geq 0\mbox{ (resp. $\leq 0$)},\ \forall\,x\in Q_*,\quad\left|\,\int_0^t \De u\big(X(s,x)\big)\,ds\,\right|\leq C_*.
\eeq
Then, $u$ is constant in a neighborhood of $x_*$.
\end{Thm}
\begin{Rem}
The boundedness condition \refe{C*}. is quite similar to the condition which permits to obtain the isotropic realizability in the torus in the absence of critical points. This approach does not imply any constraint on the dimension.
\end{Rem}
\noindent
{\bf Proof of Theorem~\ref{thm.stable}.}
Assume for example that $x_*$ is positively stable for~\refe{X}..
Define for any positive integer $n$, the function $w_n$ by
\beq\label{wn}
w_n(x):=\int_0^n \De u\big(X(s,x)\big)\,ds,\quad\mbox{for }x\in Q_*.
\eeq
Using the semigroup property $X\big(s,X(t,x)\big)=X(s+t,x)$, the function $w_n$ satisfies for any $t\geq 0$ and any $x\in Q_*$,
\beq\label{wnX}
\ba{ll}
\dis w_n\big(X(t,x)\big)-w_n(x) & \dis =\int_t^{t+n}\De u\big(X(s,x)\big)\,ds-\int_0^n \De u\big(X(s,x)\big)\,ds,
\\ \ecart
& \dis =-\int_0^t \De u\big(X(s,x)\big)\,ds+\int_n^{t+n}\De u\big(X(s,x)\big)\,ds
\\ \ecart
& \dis =-\int_0^t \De u\big(X(s,x)\big)\,ds+\int_0^t\De u\big(X(s,X(n,x))\big)\,ds.
\ea
\eeq
Taking the derivative of \refe{wnX}. with respect to $t$ at the origin, this yields
\beq\label{unw}
\nabla w_n(x)\cdot\nabla u(x)+ \De u(x)=\De u\big(X(n,x)\big),\quad\forall\,x\in Q_*.
\eeq
Hence, we get that
\beq\label{ewnu}
\div\left(e^{w_n}\nabla u\right)=e^{w_n}(\nabla w_n\cdot\nabla u+ \De u)=e^{w_n}\De u\big(X(n,\cdot)\big)\quad\mbox{in }Q_*.
\eeq
By condition \refe{C*}. the sequence $e^{w_n}$ is bounded in $L^\infty(Q_*)$, thus converges weakly-$*$ up to a subsequence to some $\si$ in $L^\infty(Q_*)$.
\par
On the other hand, by virtue of Remark~\ref{rem.stable}, for any $x\in Q_*$ the sequence $X(n,x)\in K_*$ converges to $x_*$ as the unique critical point of~$K_*$. This combined with $\De u(x_*)=0$ and the boundedness of $e^{w_n}$, implies that the sequence $e^{w_n}\De u\big(X(n,\cdot)\big)$ converges to $0$ everywhere in~$Q_*$. By \refe{C*}. it is also bounded in $L^\infty(Q_*)$. Now, integrating by parts~\refe{ewnu}., then using the weak-$*$ convergence of $e^{w_n}$ and Lebesgue's dominated convergence theorem, we get that for any $\ph\in C^\infty_c(Q_*)$,
\beq
\int_{Q_*}e^{w_n}\nabla u\cdot\nabla\ph\,dx=-\int_{Q_*}e^{w_n}\De u\big(X(n,\cdot)\big)\,\ph\,dx\limi\int_{Q_*}\si\nabla u\cdot\nabla\ph\,dx=0,
\eeq
which yields that $\si\nabla u$ is divergence free in $Q_*$. Therefore, $\nabla u$ is realizable with the non-negative conductivity $\si\in L^\infty(Q_*)$. Moreover, condition \refe{C*}. shows that $\si$ is also bounded from below by $e^{-C_*}$, so that $\si^{-1}\in L^\infty(Q_*)$. Hence, $u$ is a weak solution of $\div\left(\si\nabla u\right)=0$ in $Q_*$, where the positive conductivity $\si$ satisfies $\si,\si^{-1}\in L^\infty(Q_*)$. Moreover, by Remark~\ref{rem.stable} the point $x_*$ is a local maximum of $u$. Therefore, by the strong maximum principle for weak solutions to second-order elliptic pde's (see, {\em e.g.}, \cite{GiTr} Theorem~8.19), the function $u$ is constant in a neighborhood of $x_*$.  \cqfd
\section{An example with a non-hyperbolic point}
Consider a point $x_*$ satisfying \refe{isocri}., which is not hyperbolic in the sense that $\nabla u^2(x_*)$ has a zero determinant, and which is not stable for \refe{X}.. Generally speaking, a neighborhood of $x_*$ can be divided in several regions such that for which of them the gradient system \refe{X}. mimics either a saddle point, a sink or a source. The coexistence of these different behaviors in the phase portrait prevents $\nabla u$ from being isotropically realizable.
\par
In the sequel a point of $\RR^2$ is denoted by the coordinates $(x,y)$.
To illustrate this degenerate case consider the potential $u:(-1,1)^2\to\RR$ defined by
\beq\label{udeg}
u(x,y):={1\over 3}\left(y^3-x^3\right),\quad\mbox{for }(x,y)\in(-1,1)^2.
\eeq
The point $(0,0)$ is the unique critical point of $\nabla u$, and $\nabla^2 u(0,0)$ is the zero $(2\times 2)$ matrix. Moreover, $(0,0)$ is not a local extremum of $u$, thus it is not stable in the sense of definition~\refe{stable}.. The phase portrait representation of \reff{fig2}. shows both the stable point behavior, the sink behavior and the source behavior. Each of them appears in one of the four quadrants of the plane $\RR^2$.
\par
For the gradient of $u$ defined by \refe{udeg}., we have the following non-realizability result:
\begin{Pro}\label{pro.deg}
\hfill
\begin{itemize}
\item[$i)$] In the open set $\Om=(0,1)\times(-1,0)$ (resp. $(-1,0)\times(0,1)$), the point $(0,0)$ is positively (resp. negatively) stable.
However, $\nabla u$ is not isotropically realizable with any positive conductivity $\si\in C^1(\Om)\cap L^\infty(\Om)$.
\item[$ii)$] In the open set $\Om=(0,1)^2$ or $(-1,0)^2$, the point $(0,0)$ is not stable. Moreover, $\nabla u$ is isotropically realizable with some positive conductivity $\si_0\in C^1(\Om)\cap L^\infty(\Om)$. However, $\nabla u$ is not realizable with any positive conductivity $\si\in C^1(\Om)\cap L^\infty(\Om)$ such that $\si^{-1}\in L^\infty(\Om)$.
\end{itemize}
\end{Pro}
\begin{Rem}
In each open quadrant $\Om$ of $\RR^2$, the field $\nabla u$ is isotropically realizable with the smooth conductivity
\beq
\si(x,y):={1\over x^2 y^2}\geq 1,\quad\mbox{for }(x,y)\in\Om,
\eeq
which is bounded from below by $1$. But the function $\si$ does not belong to $L^\infty(\Om)$.
Moreover, Proposition~\ref{pro.deg} $i)$ implies that there is no conductivity having a better bound from above in $\Om=(0,1)\times(-1,0)$ or $(-1,0)\times(0,1)$.
\end{Rem}
\begin{figure}
\centering
\vskip -3.cm
\includegraphics[scale=.7]{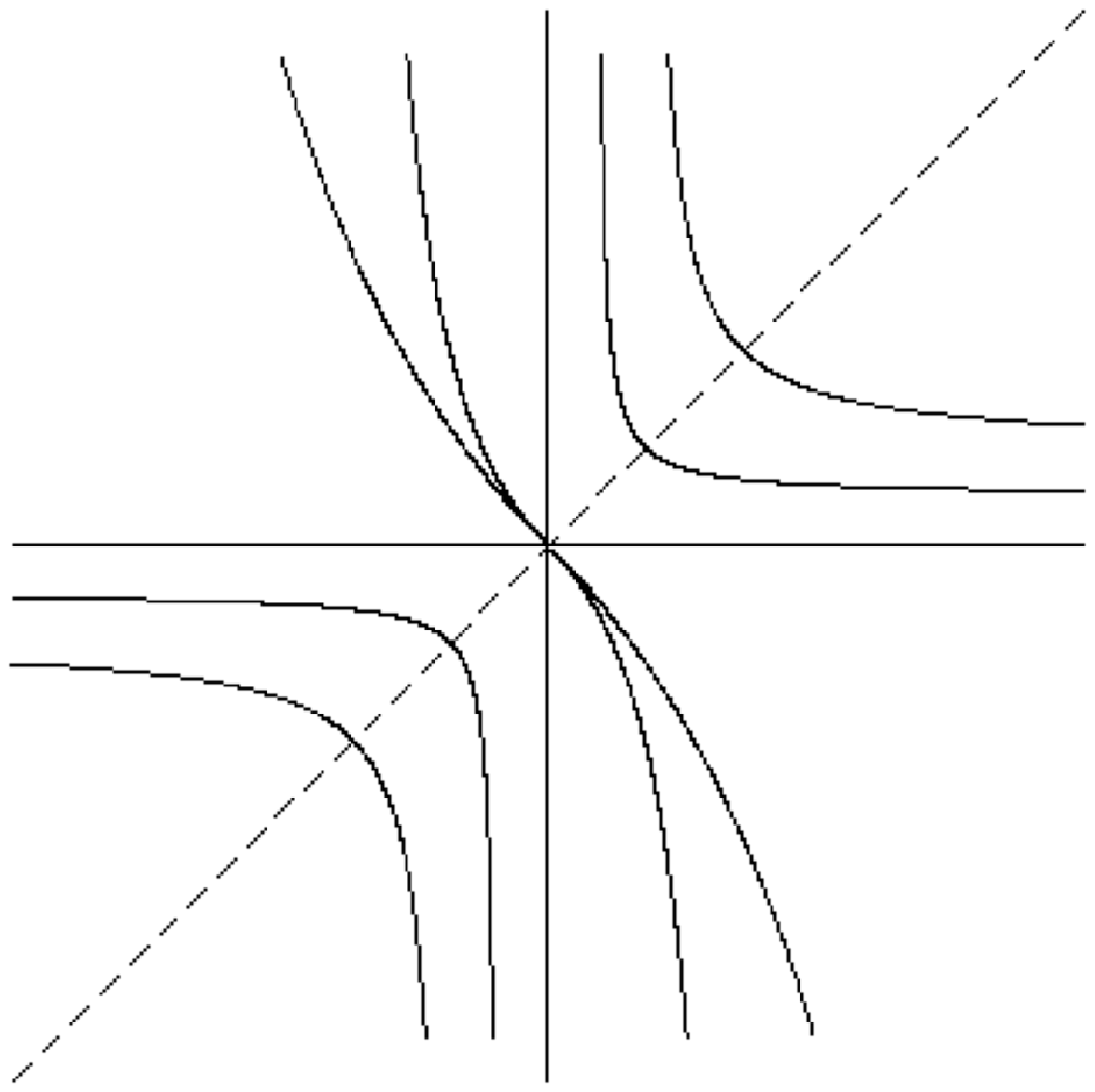}
\vskip -10.25cm
\caption{\it Phase portrait for $u={1\over 3}\left(y^3-x^3\right)$ around the point $(0,0)$, with the saddle point behavior in $\{xy>0\}$, the sink behavior in $\{x>0,y<0\}$, the source behavior in $\{x<0,y>0\}$, and the equipotential $\{u=0\}$ in dash intersecting the trajectories.}
\label{fig2}
\end{figure}
\noindent
{\bf Proof of Proposition~\ref{pro.deg}.}
\par\ms\noindent
{\it Proof of $i)$.} Let us study the case $\Om=(0,1)\times(-1,0)$. A simple computation shows that he solution of \refe{X}. is given by
\beq\label{Xdeg}
X(t,x,y)=\left({x\over 1+tx},{y\over 1-ty}\right)\in\Om,\quad\mbox{for any }(x,y)\in\Om\mbox{ and }t\geq 0,
\eeq
which converges to $(0,0)$ as $t\to\infty$. Therefore, the point $(0,0)$ is positively stable.
\par
Now, assume that there exists a positive function $\si\in C^1(\Om)\cap L^\infty(\Om)$ such that
\beq
\div\left(\si\nabla u\right)=0\;\;\mbox{in }\Om\quad\mbox{or equivalently}\quad\nabla(\ln\si)\cdot\nabla u+\De u=0\;\;\mbox{in }\Om.
\eeq
Fix $(x,y)\in\Om$. Then, integrating over $[0,t]$ the equality
\beq
{\partial\over\partial s}\left[\ln\si\big(X(s,x,y)\big)\right]=\big(\nabla(\ln\si)\cdot\nabla u\big)\big(X(s,x,y)\big)=-\De u\big(X(s,x,y)\big),
\quad\mbox{for }s\geq 0,
\eeq
it follows that
\beq\label{siDeu}
{\si(x,y)\over\si\big(X(t,x,y)\big)}=\exp\left(\int_0^t\De u\big(X(s,x,y)\big)\,ds\right),\quad\forall\,t\geq 0.
\eeq
Moreover, an easy computation using \refe{Xdeg}. yields
\beq
\int_0^t\De u\big(X(s,x,y)\big)\,ds=-2\,\ln\big[(1+tx)\,(1-ty)\big].
\eeq
The two previous formulas combined with the boundedness of $\si$ in $\Om$, imply that
\beq
0<\si(x,y)\leq{c\over(1+tx)^2\,(1-ty)^2}\;\mathop{\longrightarrow}_{t\to\infty}\;0,
\eeq
which gives a contradiction.
\par\bs\noindent
{\it Proof of $ii)$.} Let us study the case $\Om=(0,1)^2$. Following the approach of the saddle point case, for any $x,y>0$ the time $\tau=\tau(x,y)$ such that $u\big(X(\tau,x,y)\big)=0$ is given by
\beq
\tau(x,y)={x-y\over 2xy}\in\left(-{1\over 2x},{1\over 2y}\right).
\eeq
Then, the positive function $\si_0$ defined by
\beq\label{si0}
\si_0(x,y):=\exp\left(\int_0^{\tau(x,y)} \De u\big(X(s,x,y)\big)\,ds\right)={16\,x^2y^2\over(x+y)^4},\quad\mbox{for }(x,y)\in\Om,
\eeq
satisfies $\div\left(\si_0\nabla u\right)=0$ in $\Om$. Moreover, we have $\si_0\in C^1(\Om)$ and $0<\si_0\leq 16$ in $\Om$.
\par
One the other hand, assume that there exists a positive function $\si\in C^1(\Om)\cap L^\infty(\Om)$ with $\si^{-1}\in L^\infty(\Om)$, such that $\div\left(\si\nabla u\right)=0$ in $\Om$. Then, the formulas \refe{si0}. and \refe{siDeu}. with $t=\tau(x,y)$, imply that for any $(x,y)\in\Om$,
\beq\label{sisi0}
{\si(x,y)\over\si\big(X(\tau,x,y)\big)}={16\,x^2 y^2\over(x+y)^4}\,,
\eeq
where by \refe{Xdeg}. $X(\tau,x,y)={2xy\over x+y}\left(1,1\right)\in\Om$.
Therefore, due to the condition satisfied by $\si$ the left-hand side of \refe{sisi0}. is bounded from above by a positive constant on $\Om$, while the right-hand side converges to $0$ as $x\to 0$ for a fixed $y>0$. This leads to a contradiction, and concludes the proof. \cqfd
\section{Isotropic realizability in the torus}
\subsection{A conjecture and a general result}
Let $Y$ be the unit cube of $\RR^d$.
In view of the previous results and Theorem~2.17 in \cite{BMT} we may state the following conjecture on the isotropic realizability in the torus:
\begin{Conj}\label{conj1}
Let $u\in C^3(\RR^d)$, for $d\geq 2$, be a function such that $\nabla u$ is $Y$-periodic. Assume that the critical points of $u$ are isolated, and that the conditions \refe{C1Rd}. and \refe{C2Rd}. below hold. Then, $\nabla u$ is isotropically realizable in the torus with a positive conductivity $\si$ such that $\si,\si^{-1}\in L^\infty_\sharp(Y)$.
\end{Conj}
Conjecture~\ref{conj1} is a natural extension, allowing the existence of isolated critical points, of the realizability in the torus which was derived in \cite{BMT} in the absence of critical point and under the same estimate~\refe{C2Rd}.. For the moment we have not succeeded to prove this result.
However, the next result gives a partial answer to Conjecture~\ref{conj1} under the extra assumption that the trajectories are bounded:
\begin{Thm}\label{thm.globsing}
Let $u\in C^3(\RR^d)$, for $d\geq 2$, be a potential the gradient of which is $Y$-periodic, and such that for almost every $x\in\RR^d$, the trajectory $X\big([0,\infty),x\big)$ of \refe{X}. -- defined as the range of the mapping $t\geq 0\mapsto X(t,x)$ -- or the trajectory $X\big((-\infty,0],x\big)$ is bounded in $\RR^d$.
\begin{itemize}
\item[$i)$] Assume that
\beq\label{C1Rd}
\forall\,x\in\RR^d,\quad\nabla u(x)=0\ \Rightarrow\ \De u(x)=0,
\eeq
 and that there exists a constant $C>0$ such that
\beq\label{C2Rd}
\forall\,t\geq 0,\ \forall\,x\in\RR^d,\quad\left|\,\int_0^t \De u\big(X(s,x)\big)\,ds\,\right|\leq C.
\eeq
Then, $\nabla u$ is isotropically realizable in the torus with a positive conductivity $\si$ such that $\si,\si^{-1}\in L^\infty_\sharp(Y)$.
\item[$ii)$] Conversely, assume that $\nabla u$ is isotropically realizable in the torus with a positive conductivity $\si\in C^1_\sharp(Y)$. Then, conditions \refe{C1Rd}. and \refe{C2Rd}. hold true.
\end{itemize}
\end{Thm}
\begin{proof}
\hfill\par\ss\noindent $i)$ For the sake of simplicity $Y=[0,1]^d$. Consider the function $w_n$ defined by \refe{wn}.. Let $x\in\RR^d$ be such that the trajectory say $X\big([0,\infty),x\big)$ is bounded. Consider any subsequence $k_n$ of $n$ such that $X(k_n,x)$ converges to some $x_*\in\RR^d$. The point $x_*$ is necessarily a critical point of $u$ (see, {\em e.g.}, \cite{HSD} Proposition p.~206). Hence, by \refe{C1Rd}. and \refe{C2Rd}. the sequence $e^{w_{k_n}(x)}\De u\big(X(k_n,x)\big)$ converges to $0$. Therefore, the whole sequence $e^{w_n}\De u\big(X(n,\cdot)\big)$ converges to $0$ almost everywhere in $\RR^d$. By \refe{C2Rd}. this sequence is also bounded in $L^\infty(\RR^d)$. It thus follows from Lebesgue's dominated convergence theorem that $e^{w_n}\De u\big(X(n,\cdot)\big)$ converges strongly to $0$ in $L^1_{\rm loc}(\RR^d)$. However, again using estimate \refe{C2Rd}., up to a subsequence $e^{w_n}$ converges weakly-$*$ in $L^\infty(\RR^d)$ to some positive function $\si_0$, with $\si_0^{-1}\in L^\infty(\RR^d)$. Therefore, passing to the limit $n\to\infty$ in \refe{ewnu}., we get that $\div\left(\si_0\nabla u\right)=0$ in the distributions sense on $\RR^d$. Finally, up to extract a new subsequence the average $\si_n$ defined for any positive integer $n$ by
\[
\sigma_n(x):={1\over(2n+1)^d}\sum_{k\in \ZZ^d\cap[-n,n]^d}\sigma_0(x+k),\quad\mbox{for }x\in\RR^d,
\]
converges weakly-$*$ in $L^\infty(\RR^d)$ to some $Y$-periodic function $\si$ which does the job as in the non-critical case. We refer to \cite{BMT} for further details about this average argument.
\par\ms\noindent $ii)$ The proof is quite similar to the part $ii)$ of Theorem~\ref{thm.stable}.
\end{proof}
\begin{Rem}\label{rem.d=2}
The dimension two is quite particular in the framework of Conjecture~\ref{conj1} and Theorem~\ref{thm.globsing}. Indeed, consider a smooth potential $u$ defined in $\RR^2$ such that $\nabla u$ is $Y$-periodic and is not identically zero in $\RR^2$. Also assume that there exists a smooth $Y$-periodic conductivity $\si>0$ such that $\div\left(\si\nabla u\right)=0$ in $\RR^2$. Then, thanks to Proposition~2 in \cite{AlNe} the function $u$ has no critical point in $\RR^2$. Moreover, the trajectories $X\big([0,\infty),x\big)$ and $X\big((-\infty,0],x\big)$ are unbounded, since any limit point of a bounded trajectory is a critical point of the gradient (see \cite{HSD} Proposition p.~206). Therefore, the absence of critical point agrees with the unboundedness of the trajectories in the two-dimensional smooth periodic case.
\end{Rem}
The following example illustrates the strong connection between the isotropic realizability and estimate \refe{C2Rd}..
\subsection{Application}
Let $u_i$, for $i\in\{1,\dots,d\}$, be a function in $C^2(\RR)$ such that its derivative $u_i'$ is $T_i$-periodic for some $T_i>0$ and has only isolated roots $t\in\RR$ also satisfying $u_i''(t)=0$.
Define $u\in C^2(\RR^d)$ by
\beq\label{uui}
u(x):=\sum_{i=1}^d u_i(x_i),\quad\mbox{for }x=(x_1,\dots,x_d)\in\RR^d.
\eeq
The gradient field $\nabla u$ is $Y$-periodic with $Y:=\prod_{i=1}^d[0,T_i]$, the critical points of $u$ are isolated and condition \refe{C1Rd}. is satisfied.
\par
Then, we have the following result:
\begin{Pro}\label{pro.ui}
The trajectories of the gradient system \refe{X}. are bounded if and only if each function $u'_i$ vanishes in $\RR$.
Moreover, the following assertions are equivalent:
\[
\mbox{estimate \refe{C2Rd}. holds},
\]
\beq\label{ui'­0}
\forall\,x\in\RR^d,\quad\prod_{i=1}^d u_i'(x_i)\neq 0,
\eeq
\beq\label{siui'}
\nabla u\mbox{ is isotropically realizable with a positive function }\si\in C^1_\sharp(Y).
\eeq
\end{Pro}
\begin{Exa}
The function $u(x):=x_1-\cos(2\pi x_2)$ from \cite{BMT} is a particular case of this framework. It clearly does not satisfy condition~\refe{ui'­0}.. We thus recover that $\nabla u$ is not isotropically realizable in the torus.
\end{Exa}
\noindent
{\bf Proof of Proposition~\ref{pro.ui}.}
First, let us determine conditions for having bounded trajectories of system \refe{X}.:
\par
Let $i\in\{1,\dots,d\}$ and $x\in\RR^d$.
If the derivative $u_i'$ vanishes in $\RR$, then by periodicity it has an infinite number of zeros which are isolated by assumption. Consider two consecutive zeros $\al_i<\be_i$ of $u_i'$. Assume that $x_i$ belongs to $(\al_i,\be_i)$. Then, the $i$-th component $X_i(\cdot,x)$ of the trajectory $X(\cdot,x)$ remains in $(\al_i,\be_i)$. Indeed, $X_i(\cdot,x)$ cannot meet one of the sets $\{\al_i\}$ or $\{\be_i\}$ which are themselves two trajectories. 
Conversely, if the derivative $u_i'$ does not vanish in $\RR$, then $X_i(\cdot,x)$ is given by
\beq\label{Xi}
X_i(t,x)=v_i^{-1}\big(t+v_i(x_i)\big),\quad\mbox{where}\quad v_i(y):=\int_{x_i}^y {ds\over u_i'(s)}\;\mathop{\longrightarrow}_{|y|\to\infty}\;\pm\infty,
\eeq
and $v_i^{-1}$ is the reciprocal function of $v_i$. Therefore, the trajectory $X_i(\cdot,x)$ is not bounded.
\par\ms\noindent
$\underline{\refe{C2Rd}.\Rightarrow\refe{ui'­0}.}$: For $i\in\{1,\dots,n\}$ and for $x\in\RR^d$, we have $X_i(\RR,x)=(\al_i,\be_i)$, where either $\al_i<\be_i$ are two consecutive zeros of $u_i'$ or $(\al_i,\be_i)=\RR$ if $u_i'$ does not vanish in $\RR$. Let $v_i$ be a primitive of $1/u_i'$ on $(\al_i,\be_i)$. Making the change of variables $y=v_i^{-1}\big(s+v_i(x_i)\big)$, we get that for $t\geq 0$,
\beq
\ba{ll}
\dis \int_0^t \De u\big(X(s,x)\big)\,ds & \dis =\sum_{i=1}^d\int_0^t u_i''\left((v_i^{-1}\big(s+v_i(x_i)\big)\right)ds
=\sum_{i=1}^d\int_{x_i}^{v_i^{-1}(t+v_i(x_i))}{u_i''(y)\over u_i'(y)}\,dy
\\ \ecart
& \dis =\ln\left(\prod_{i=1}^d{\left|u_i'\left(v_i^{-1}\big(t+v_i(x_i)\big)\right)\right|\over\left|u_i'(x_i)\right|}\right)
\ea
\eeq
which is bounded independently of $t$ and $x$. Or equivalently, there exists a constant $c>0$ such that
\beq\label{estui'}
\forall\,t\geq 0,\ \forall\,x\in\RR^d,\quad c\leq\prod_{i=1}^d{\left|u_i'\left(v_i^{-1}\big(t+v_i(x_i)\big)\right)\right|\over\left|u_i'(x_i)\right|}\leq c^{-1}.
\eeq
However, since by \refe{Xi}. $v_i^{-1}\big(t+v_i(x_i)\big)$ converges to $\al_i$ or $\be_i$ as $t\to\infty$, estimate \refe{estui'}. implies that each function $u_i'$ cannot vanish in $\RR$.
\par\ms\noindent
$\underline{\refe{ui'­0}.\Rightarrow\refe{siui'}.}$: It is easy to see that the conductivity
\beq
\si(x):=\left|\,\prod_{i=1}^d u_i'(x_i)\,\right|^{-1}>0,\quad\mbox{for }x=(x_1,\dots,x_d)\in\RR^d,
\eeq
does the job.
\par\ms\noindent
$\underline{\refe{siui'}.\Rightarrow\refe{C2Rd}.}$: It is a straightforward consequence of Theorem~\ref{thm.globsing} $ii)$.
\cqfd
\par\bs\noindent
{\bf Acknowledgment.} The author is very grateful to G.W. Milton for stimulating discussions on the topic.
\appendix
\section{A global rectification theorem in dimension two}
\begin{Thm}\label{thm.globrect}
Let $u$ be a function in $C^3(\RR^2)$ satisfying condition \refe{Du­0}.. Then, there exists a function $\Phi$ in $C^1(\RR^2;\RR^2)$ such that the jacobian matrix $\nabla\Phi$ is invertible in $\RR^2$ and
\beq\label{DPhi}
\nabla\Phi\,\nabla u=e_1=\begin{pmatrix} 1 \\ 0 \end{pmatrix}\quad\mbox{in }\RR^2.
\eeq
\end{Thm}
\begin{Rem}
The function $\Phi$ is a local but is not necessarily a global diffeomorphism on $\RR^2$. However, the gradient system \refe{X}. $X'=\nabla u(X)$ is changed into the system $Y'=e_1$ by applying the chain rule to $Y=\Phi(X)$. Therefore, the mapping $\Phi$ rectifies globally in $\RR^2$ the gradient field $\nabla u$ into the vector $e_1$.
\end{Rem}
\begin{proof}
Following Theorem~2.15 in \cite{BMT}, by condition \refe{Du­0}.  there exists a unique function $\tau$ in $C^1(\RR^2)$ which satisfies the equation $u\big(X(\tau(x),x)\big)=0$ for any $x\in\RR^2$. The semi-group property satisfied by the flow $X(t,x)$ combined with the uniqueness of $\tau$ also yields
\beq\label{tauX}
\tau\big(X(t,x)\big)=\tau(x)-t,\quad\forall\,(t,x)\in\RR\times\RR^2.
\eeq
Hence, taking the derivative of \refe{tauX}. with respect to $t$, we have
\beq
{\partial\over\partial t}\big[\tau\big(X(t,x)\big)\big]=\nabla\tau\big(X(t,x)\big)\cdot\nabla u\big(X(t,x)\big)=-1,\quad\forall\,x\in\RR^2,
\eeq
which implies that
\beq\label{DtauDu}
\nabla\tau\cdot\nabla u=-1\quad\mbox{in }\RR^2.
\eeq
On the other hand, again by Theorem~2.15 of \cite{BMT} the function $\si$ defined by
\beq\label{sitau}
\si(x):=\exp\left(\int_0^{\tau(x)}\De u\big(X(s,x)\big)\,ds\right),\quad\mbox{for }x\in\RR^2,
\eeq
solves the equation $\div\left(\si\nabla u\right)=0$ in $\RR^2$. It follows that the function $v$ defined by
\beq\label{v}
v(x):=\int_0^{x_1}\si(s,x_2)\,{\partial u\over\partial x_2}(s,x_2)\,ds-\int_0^{x_2}\si(0,t)\,{\partial u\over\partial x_2}(0,t)\,dt,
\quad\mbox{for }x\in\RR^2,
\eeq
is a stream function in $C^1(\RR^2)$ associated with the divergence free current $\si\nabla u$, hence
\beq\label{siDuDw}
\si\nabla u=\nabla^\perp v\quad\mbox{and thus}\quad\nabla v\cdot\nabla u=0\quad\mbox{in }\RR^2.
\eeq
Therefore, by \refe{DtauDu}. and \refe{siDuDw}. the function $\Phi=\left(-\tau,v\right)$ belongs to $C^1(\RR^2;\RR^2)$ and satisfies equality \refe{DPhi}..
\par
Finally, assume that the matrix $\nabla\Phi(x)$ is not invertible for some $x\in\RR^2$. Then, since $\nabla\tau(x)\neq 0$, there exists $\al\in\RR$ such that $\nabla v(x)=\al\,\nabla\tau(x)$. Hence, by \refe{DtauDu}. and \refe{siDuDw}. we get that
$0=\big(\nabla v(x)-\al\,\nabla\tau(x)\big)\cdot\nabla u(x)=\al$, which yields $\nabla v(x)=0$. However, the condition \refe{Du­0}. combined with the first equality of \refe{siDuDw}. implies that $\nabla v(x)\neq 0$, which leads to a contradiction. Therefore, $\nabla\Phi$ is invertible in $\RR^2$ and $\Phi$ is a local diffeomorphism in $\RR^2$. It is not clear that $\Phi$ is one-to-one to conclude that $\Phi$ is diffeomorphism from $\RR^2$ onto $\RR^2$.
\end{proof}
\end{document}